\documentclass{article}

\usepackage{geometry}                
\usepackage{graphicx}
\usepackage{amssymb}
\usepackage{epstopdf}
\usepackage{pgf,tikz}
\usepackage{amsmath}
\usepackage{tikz}
\usepackage{graphicx}
\usepackage{mathtools}
\usepackage{graphicx}
\usepackage{epstopdf}
\usepackage{fix-cm}
\usepackage{geometry}                
\usepackage{graphicx}
\usepackage{amssymb}
\usepackage{epstopdf}
\usepackage{pgf,tikz}
\usepackage{latexsym}
\usepackage{amsmath,amsfonts}
\usepackage{graphicx}
\usepackage{geometry}
\usepackage{graphicx}
\usepackage{multirow}
\usepackage{amssymb}
\usepackage{epstopdf}
\usepackage{pgf,tikz}
\usepackage{latexsym}
\usepackage[english]{babel}
\usepackage{dcolumn}
\usepackage{color}

\usepackage{multirow}
\usepackage{amsmath,amsfonts,amsthm}
\usepackage{mathtools,url,tikz}
\usepackage{graphics}
\usepackage{graphicx}

\newtheorem{theorem}{Theorem}[section]

\newtheorem{corollary}[theorem]{Corollary}

\newtheorem{lemma}[theorem]{Lemma}

\newcommand{\ignore}[1]{}
\newcommand{\bk}{\mathbf{k}}
\newcommand{\bh}{\mathbf{h}}
\newcommand{\uuf}{f}

\newcommand{\bK}{{\mathbf K}}
\newcommand{\oR}{{\mathbb R}}

\newcommand{\oN}{{\mathbb N}}

\begin{document}

\title{Worst-case examples for Lasserre's measure--based hierarchy for polynomial optimization {on the hypercube}}

\author{Etienne de Klerk \thanks{Tilburg University and Delft University of Technology, \texttt{E.deKlerk@uvt.nl}}
 \and Monique Laurent \thanks{Centrum Wiskunde \& Informatica (CWI), Amsterdam and Tilburg University, \texttt{monique@cwi.nl}}}

\maketitle

\begin{abstract}
We study the convergence rate of a  hierarchy of  upper bounds for polynomial optimization problems,
proposed by Lasserre [{\em SIAM J. Optim.} $21(3)$ $(2011)$, pp. $864-885$], and a related hierarchy by De Klerk, Hess and Laurent
[{\em SIAM J. Optim.} $27(1)$, $(2017)$ pp. $347-367$].
{For polynomial optimization over the hypercube we show a refined convergence analysis for the first hierarchy. We also show lower bounds on the convergence rate for both hierarchies on a class of examples. These lower bounds match the upper bounds and thus establish the true rate of convergence on these examples.}
Interestingly, 
{these convergence rates} are determined by the distribution of extremal zeroes
of certain families of orthogonal polynomials.
\end{abstract}

\noindent
{\bf Keywords:} {Polynomial optimization; Semidefinite optimization;  Lasserre hierarchy; extremal roots of orthogonal polynomials; Jacobi polynomials} \\
{\bf AMS classification:} {90C22; 90C26; 90C30} \\

\section{Introduction }\label{secintro}

We consider the problem of minimizing a polynomial $f:\oR^n\to\oR$ over a compact set  $\mathbf{K}\subseteq \oR^n$.
That is, we consider the problem of computing the parameter:

\begin{equation*}\label{fmink}
f_{\min,\mathbf{K}}:= \min_{x\in \mathbf{K}}f(x).
\end{equation*}
 We recall the following reformulation for $f_{\min,\mathbf{K}}$, established by Lasserre \cite{Las11}:

\begin{equation*}\label{fminkreform2}
f_{\min,\mathbf{K}}=\inf_{\sigma\in\Sigma[x]}\int_{\mathbf{K}}\sigma(x)f(x)d\mu(x) \ \ \mbox{s.t. $\int_{\mathbf{K}}\sigma(x)d\mu(x)=1$,}
\end{equation*}
where $\Sigma[x]$ denotes the set of sums of squares of polynomials, and $\mu$ is a signed Borel measure supported on $\mathbf{K}$.
\smallskip
\noindent
Given an integer $d\in \oN$, by bounding the degree of the polynomial $\sigma\in \Sigma[x]$ by $2d$, Lasserre \cite{Las11} defined  the parameter:

\begin{eqnarray}\label{fundr}
\underline{f}^{(d)}_{\mathbf{K}}:=\inf_{\sigma\in\Sigma[x]_d}\int_{\mathbf{K}}\sigma(x)f(x)d\mu(x) \ \ \mbox{s.t. $\int_{\mathbf{K}}\sigma(x)d\mu(x)=1$,}
\end{eqnarray}
where $\Sigma[x]_d$ consists of the polynomials in $\Sigma[x]$ 
with degree at most $2d$.

The inequality  \smash{$f_{\min,\mathbf{K}}\le\underline{f}^{(d)}_{\mathbf{K}}$} holds for all $d\in\oN$ and, in view  of the identity   (\ref{fminkreform2}), it follows that the sequence \smash{$\underline{f}^{(d)}_{\mathbf{K}}$} converges to $f_{\min,\mathbf{K}}$ as $d\rightarrow \infty$.
De Klerk and  Laurent \cite{DKL MOR}
established the following rate of convergence for the sequence
$\underline{f}^{(d)}_{\mathbf{K}}$, when $\mu$ is  the Lebesgue measure and $\bK$ is a convex body.

\begin{theorem}\cite{DKL MOR}
 \label{thm:dKLS}
Let $f\in \oR[x]$,  $\mathbf{K}$ a convex body, and $\mu$ the Lebesgue measure on $\mathbf{K}$. There exist constants $C_{f,\bK}$ (depending only on $f$ and $\bK$) and $d_\bK\in \oN$ (depending only on $\bK$) such that
\begin{equation}\label{thmmaineq2}
\underline{f}^{(d)}_{\mathbf{K}}-f_{\min,\mathbf{K}} \le {C_{f,\bK} \over  {d}}\ \ \text{ for all } d\ge d_\bK.
\end{equation}
 That is, the  following asymptotic convergence rate holds: $\underline{f}^{(d)}_{\mathbf{K}}-f_{\min,\mathbf{K}} = O\left( {1\over  d}\right).$
\end{theorem}
This result was an improvement on an earlier result by De Klerk, Laurent and Sun \cite[Theorem~3]{KLS MPA}, who showed a convergence rate in \smash{$O(1/\sqrt d)$} (for $\bK$ convex body or, more generallly, compact under a mild assumption).

As explained in \cite{Las11}  the parameter $\underline{f}^{(d)}_{\mathbf{K}}$ can be computed using semidefinite programming,
assuming  one knows  the (generalised)  moments of the  measure $\mu$ on $\bK$ with respect to some polynomial basis. Set
\begin{equation*}\label{mack}
m_{\alpha}(\mathbf{K}):=\int_{\mathbf{K}}b_{\alpha}(x)d\mu(x), \;
m_{\alpha,\beta}(\mathbf{K}):=\int_{\mathbf{K}}b_{\alpha}(x)b_{\beta}(x)d\mu(x) \ \ \ \mbox{ for } \alpha, \beta\in \oN^n,
\end{equation*}
where the polynomials $\{b_\alpha\}$ form a basis for the space $\oR[x_1,\ldots,x_n]_{2d}$ of polynomials of degree at most $2d$, indexed by
$N(n,2d)=\{\alpha\in \oN^n: \sum_{i=1}^n\alpha_i\le 2d\}$.
For example, the standard monomial basis in $\oR[x_1,\ldots,x_n]_{2d}$ is $b_\alpha(x) = x^\alpha := \prod_{i=1}^n x_i^{\alpha_i}$ for $\alpha\in N(n,2d)$, and then
$m_{\alpha,\beta}(\mathbf{K})= m_{\alpha+\beta}(\mathbf{K})$.
\smallskip
If $f(x)=\sum_{\beta\in N(n,d_0)}f_{\beta}b_{\beta}(x)$ has degree $d_0$, and
writing  $\sigma\in\Sigma[x]_{d}$ as $\sigma(x)=\sum_{\alpha\in N(n,2d)}\sigma_{\alpha}b_{\alpha}(x)$,  then the parameter $\underline{f}^{(d)}_{\mathbf{K}}$ in
 (\ref{fundr}) can be computed as follows:
\begin{eqnarray}\label{eqSDP}
\underline{f}^{(d)}_{\mathbf{K}}&=&\min\sum_{\beta\in N(n,d_0)}f_{\beta}\sum_{\alpha\in N(n,2d)}\sigma_{\alpha}m_{\alpha,\beta}(\mathbf{K})\label{fundr2}\\
 & &\mbox{ s.t. } \ \ \sum_{\alpha\in N(n,2d)}\sigma_{\alpha}m_{\alpha}(\mathbf{K})=1,\nonumber\\
&&\ \ \ \ \ \ \ \sum_{\alpha\in N(n,2d)}\sigma_{\alpha}b_{\alpha}(x)\in\Sigma[x]_d.\nonumber
\end{eqnarray}
Since the sum-of-squares condition on $\sigma$ may be written as a linear matrix inequality, this is a semidefinite program.
In fact, since the program (\ref{eqSDP})   has only one linear equality constraint, using semidefinite programming duality it can  be rewritten as a generalised eigenvalue problem.
In particular,
\smash{  $\underline{f}_{\mathbf{K}}^{(d)}$} {is equal to the   the smallest  generalised eigenvalue} of the system:
\[
Ax = \lambda Bx \quad \quad\quad (x \neq 0),
\]
where the symmetric matrices $A$ and $B$ are of order ${n + d \choose d}$ with rows and columns  indexed by $N(n,d)$,
and
\begin{equation}
\label{matrices A and B}
A_{\alpha, \beta} = \sum_{\delta \in N(n,d_0)} f_\delta \int_{\mathbf{K}} b_{\alpha}(x) b_{\beta}(x) b_{\delta}(x) d\mu(x),
\quad B_{\alpha, \beta} = \int_{\mathbf{K}} b_\alpha(x)b_{\beta}(x)d\mu(x) \quad \text{ for } \alpha, \beta \in {N}(n,d).
\end{equation}
For more details, see \cite{Las11,KLS MPA}. 
In particular, if the basis $\{b_\alpha\}$ is orthonormal
with respect to the measure $\mu$, then $B$ is the identity matrix, and \smash{$\underline{f}_{\mathbf{K}}^{(d)}$} is
the smallest eigenvalue of the above matrix $A$. For further reference we summarize this result, {which will play a central role in our approach.}

\begin{lemma}\label{lemsummarize}
Assume $\{b_\alpha:\alpha\in N(n,2d)\}$ is a basis of the space $\oR[x_1,\ldots,x_n]_{2d}$, which is orthonormal w.r.t. the measure $\mu $ on $K$, i.e.,
$\int_K b_\alpha(x)b_{\beta}(x)d\mu(x)=\delta_{\alpha,\beta}$. Then the parameter \smash{$\underline{f}^{(d)}_{\mathbf{K}}$} is equal to the smallest eigenvalue of the matrix $A$ in (\ref{matrices A and B}).
\end{lemma}
Under the conditions of the lemma, note in addition that, if the vector $u=(u_\alpha)_{\alpha\in N(n,d)}$ is an eigenvector of the matrix $A$ in (\ref{matrices A and B}) for its smallest eigenvalue, then the (square) polynomial
$\sigma(x)=(\sum_{\alpha\in N(n,d)}u_\alpha x^\alpha)^2$ is an optimal density function for the parameter \smash{$\underline{f}^{(d)}_{\mathbf{K}}$}.

\subsubsection*{Related hierarchy by De Klerk, Hess and Laurent}
For the hypercube $\bK = [-1,1]^n$, De Klerk, Hess and Laurent \cite{DHL SIOPT} considered a variant on the Lasserre hierarchy \eqref{fundr},
where the density function $\sigma$ is allowed to take the more general form
\begin{equation}\label{eqSch}
\sigma(x) =  \sum_{I \subseteq \{1,\ldots,n\}} \sigma_I(x) \prod_{i\in I} (1-x_i^2)
\end{equation}
and the polynomials $\sigma_I$ are sum-of-squares polynomials with degree at most $2d-2|I|$ (to ensure  that the degree of $\sigma$ is at most $2d$), and $I = \emptyset$ is included in the summation.
Moreover the measure $\mu$ is fixed to be
\begin{equation}\label{eqmu}
d\mu(x) = \left(\prod_{i=1}^n  \sqrt{1-x_i^2} \right)^{-1}dx_1\cdots dx_n.
\end{equation}
As we will recall below, this measure is associated with the Chebyshev  orthogonal polynomials.
We let \smash{$\uuf^{(d)}$} denote the parameter\footnote{We drop the dependence on $\bK$ which is implictly selected to be the box $[-1,1]^n$.}  obtained by using in (\ref{fundr}) these choices (\ref{eqSch})  of density functions $\sigma(x)$ and (\ref{eqmu}) of measure $\mu$. 
By construction, we have
$$f_{\min,\bK} \le \uuf^{(d)}\le  \underline{f}^{(d)}_{\mathbf{K}}.$$
De Klerk, Hess and Laurent \cite{DHL SIOPT} proved a stronger  convergence rate for the bounds $\uuf^{(d)}$.

\begin{theorem}\cite{DHL SIOPT}
\label{theoDKHL}
Let $f\in\oR[x]$ be a polynomial and $\bK=[-1,1]^n.$ We have
$$\uuf^{(d)}-f_{\min,\bK} =  O\left({1\over d^2}\right).$$
\end{theorem}

\subsubsection*{Contribution of this paper}
{In this paper we investigate the rate of convergence of the hierarchies  \smash{$\underline{f}^{(d)}_{\mathbf{K}}$} and
\smash{$\uuf^{(d)}$} to $f_{\min,\bK}$ for the case of the box $\bK=[-1,1]^n$.  The above discussion  raises naturally the following questions:
\begin{itemize}
\item[$\bullet$] Is the sublinear convergence rate $\uuf^{(d)}-f_{\min,\bK} =  O\left({1\over d^2}\right)$ tight, or can this result be improved?
\item[$\bullet$] Does this convergence rate extend to the Lasserre  bounds, where we restrict to sums-of-squares density functions?
\end{itemize}
We give a positive answer to both questions.
Regarding the first question we show that the convergence rate is   $\Omega(1/d^2)$  when $f$ is a linear polynomial, 
which implies that the convergence analysis in Theorem \ref{theoDKHL} for the bounds \smash{$\uuf^{(d)}$} is tight.
This relies on the eigenvalue reformulation of the bounds  (from Lemma \ref{lemsummarize}) and an additional link to the extremal zeros of the associated  Chebyshev 
polynomials. We also show that the same  lower bound holds for the convergence rate of the Lasserre  bounds
\smash{$\underline{f}^{(d)}_{\mathbf{K}}$} when considering measures on the hypercube corresponding to general Jacobi polynomials.}

{Regarding the second question we show that also the Lasserre bounds have a $O(1/d^2)$ convergence rate when using the Chebyshev type measure from (\ref{eqmu}). The starting point is again the reformulation from Lemma \ref{lemsummarize} in terms of eigenvalues, combined with some further analytical  arguments.}

{The paper is organised as follows. In Section \ref{secprel} we group preliminary results about orthogonal polynomials and their extremal roots. Then, in Section \ref{secLas} we analyse the convergence rate of the Lasserre bounds
 \smash{$\underline{f}^{(d)}_{\mathbf{K}}$} when $f$ is a linear polynomial
and, in Section \ref{secDKHL}, we analyse the bounds \smash{$\uuf^{(d)}$}.
In both cases we show a $\Omega(1/d^2)$ lower bound. In Section \ref{secupper} we show a $O(1/d^2)$
upper bound for   the convergence rate of the Lasserre bounds  \smash{$\underline{f}^{(d)}_{\mathbf{K}}$}, and this analysis is tight  in view of the previously shown lower bounds.}

\subsubsection*{Notation}
We recap here some notation that is used throughout.
For an integer $d\in\oN$, $\oR[x]_d$ denotes the set of $n$-variate polynomials in the variables $x=(x_1,\ldots,x_n)$ with degree at most $d$ and $\Sigma[x]_d$ denotes the set of polynomials with degree at most $2d$ that can be written as a sum of squares of polynomials.

We use the classical Landau notation. For two functions $f,g:\oN\rightarrow  \oR_+$,  the notation $f(n)=O(g(n))$ (resp., $f(n)=\Omega(g(n))$, $f(n)=o(g(n))$) means
$\limsup_{n\to\infty} f(n)/g(n) <\infty$ (resp., $\liminf_{n\to\infty} f(n)/g(n)>\infty$, $\lim_{n\to\infty}f(n)/g(n) =0$), and  $f(n)=\Theta(g(n))$
means $f(n)=O(g(n))$ and $f(n)=\Omega(g(n))$.
We also use this notation when $f,g$ are functions of a continuous variable $x$ and we want to indicate the behavior of $f(x)$ and $g(x)$ in the neighbourhood of a given scalar $x_0$ when $x\to x_0$.
So, $f(x)=O(g(x))$ as $x\to x_0$ means $\limsup_{x\to x_0} f(x)/g(x) <\infty$, etc.

\section{Preliminaries on orthogonal polynomials}\label{secprel}
In what follows we review some known facts on classical orthogonal polynomials that we need for our treatment.
Unless we give detailed references, the relevant results may be found in the classical text by Szeg\"o \cite{Szego_1975}
(see also \cite{Gautsch}).

We consider families of univariate polynomials $\{p_k(x)\}$ $(k=0,1,\ldots,d)$,  that satisfy a three-term
recursive relation of the form:
\begin{equation}
\label{eq:recursion}
xp_k(x) = a_kp_{k+1}(x) + b_kp_k(x) + c_kp_{k-1}(x) \quad\quad (k=1,\ldots,d-1),
\end{equation}
where $p_0$ is a constant, 
$p_1(x) = (x-b_0)p_0/a_0$, and $a_k$, $b_k$ and $c_k$ are
real values that satisfy $a_{k-1}c_k >0$ for $k=1,\ldots,d-1$. If we set $c_0=0$ then relation (\ref{eq:recursion}) also holds for $k=0$).

\noindent Defining the $k\times k$ tri-diagonal matrix
\begin{equation}
\label{matrix_Ak}
A_k :=
\left(
  \begin{array}{ccccc}
    b_0 & a_0 & 0 & \cdots & 0 \\
   c_1 & b_1 & a_1 &  & 0 \\
    0 & \ddots & \ddots & \ddots &  \\
    \vdots &  & c_{k-2} & b_{k-2} &a_{k-2} \\
    0 & 0 & \cdots & c_{k-1} &b_{k-1} \\
  \end{array}
\right),
\end{equation}
one has the classical relation:
\begin{equation}
\label{eq:det_roots}
\left(\prod_{j=0}^{k-1} a_j\right) p_k(x) = \det(xI_k - A_k){p_0} \quad \text{ for } k=1,\ldots,d,
\end{equation}
which can be easily verified using induction on $k\ge 1$  and the relation (\ref{eq:recursion}) (see, e.g., \cite{Ismail_Li_1992}).
Therefore,  the roots of the polynomial $p_k$ are precisely the eigenvalues of the matrix $A_k$ in (\ref{matrix_Ak}). 

\medskip
Recall that the polynomials $p_k$ $(k=0,1,\ldots,d)$ are {\em orthogonal with respect to a
 weight function} $w:[-1,1]\rightarrow \mathbb{R}$, that is continuous and positive on $(-1,1)$,
 if
 \[
 \langle p_i,p_j\rangle := \int_{-1}^1 p_i(x)p_j(x) w(x)dx = 0 \quad \mbox{ for all  \ $i \neq j$}.
  \]
  We denote by $\hat p_k:=p_k/\sqrt{\langle p_k,p_k\rangle}$ the corresponding normalized polynomial, so that $\langle \hat p_k,\hat p_k\rangle =1$.

As is well known, if the polynomials $p_k$ are degree $k$ polynomials that are pairwise orthogonal with respect to such a  weight function then they  satisfy a three-terms recurrence relation of the form (\ref{eq:recursion}) (see, e.g., \cite[\S 1.3]{Gautsch}). Of course, the corresponding orthonormal polynomials $\hat p_k$ also satisfy such a three-terms recurrence relation (for different scaled parameters $a_k,b_k,c_k$).

By taking  the inner product of both sides in (\ref{eq:recursion}) with $p_{k-1}$ and  $p_{k+1}$ one gets the relations $c_k \langle p_{k-1},p_{k-1}\rangle =\langle p_k,xp_{k-1}\rangle$ and
$a_k \langle p_{k+1},p_{k+1}\rangle= \langle p_{k+1},xp_k\rangle$, which imply $c_k \langle p_{k-1},p_{k-1}\rangle
=a_{k-1}\langle p_k,p_k\rangle$ and thus $a_{k-1}c_k>0$. Moreover, when considering the recurrence relations associated with the orthonormal  polynomials $\hat p_k$, we have $a_{k-1}=c_k$ for any $k\ge 1$, i.e., the matrix $A_k$ in (\ref{matrix_Ak}) is symmetric.
We will use later the following  fact.

\begin{lemma}\label{lem_Ak}
Let $\{\hat p_k\}$ be orthonormal polynomials for the measure $d\mu(x)=w(x)dx$ on $[-1,1]$, where $w(x)$ is continuous and positive on $(-1,1)$, and assume they satisfy the three-terms recurrence relation (\ref{eq:recursion}).
Then,  the matrix   
\begin{equation}\label{eqhatA}
\left( \langle x\hat p_i,\hat p_j\rangle=  \int_{-1}^1 x\hat p_i(x)\hat p_j(x)w(x)dx\right)_{i,j=0}^{k-1}
\end{equation}
is equal to the matrix  $A_k$  in (\ref{matrix_Ak}). In particular, its smallest eigenvalue  is the smallest root of the polynomial $p_k$.
\end{lemma}

\begin{proof}
Using the recurrence relation (\ref{eq:recursion}) we obtain
\begin{eqnarray*}
\langle x \hat p_i,\hat p_j\rangle       &=& \langle a_i\hat p_{i+1}+b_i \hat p_i+c_i \hat p_{i-1},\hat p_j\rangle
  \\
       &=& \left\{
       \begin{array}{ll}
       a_i  & \mbox{if $j=i+1$} \\
       b_i & \mbox{if $j=i$} \\
       c_i  & \mbox{if $j=i-1$} \\
       0 & \mbox{otherwise}.
\end{array}\right.
\end{eqnarray*}
Hence the matrix in (\ref{eqhatA}) is equal to $A_k$ and the last claim follows from (\ref{eq:det_roots}).
\end{proof}

It  is also known that the roots of $p_k$ are all real, simple and lie in $(-1,1)$, and  that they interlace the roots of $p_{k+1}$ (see, e.g., \cite[\S 1.2]{Gautsch}). In what follows we will use  the smallest (and largest) roots to give closed-form expressions for  the bounds \smash{$ \underline f^{(d)}_\bK$ and $\uuf^{(d)}$} in some examples.

\medskip
We now recall several classical univariate orthogonal polynomials on the interval $[-1,1]$ and some information on their smallest roots.

\subsection*{Chebyshev polynomials}
We will use the  univariate   Chebyshev polynomials (of the first kind),
 defined by:
\begin{equation}\label{eqTUn}
T_k(x)= \cos(k\arccos (x)),\ \ \text{ for } \ x\in [-1,1],\ k = 0,1,\ldots.
\end{equation}
They satisfy the following three-terms recurrence relationships:
\begin{equation}\label{eqTnrec}
T_0(x)=1,\ 
T_1(x)=x, \ T_{k+1}(x)= 2xT_k(x)-T_{k-1}(x) \ \ \text{ for } \ k\ge 1.
\end{equation}
The Chebyshev polynomials are orthogonal with respect to the weight function
  $w(x) = \frac{1}{ \sqrt{1-x^2}}$
and the roots of $T_k$ are given by
\begin{equation}\label{rootTk}
\cos\left( \frac{2i-1}{2k}\pi \right) \quad \text{ for } i = 1,\ldots,k.
\end{equation}

\subsection*{Jacobi polynomials}
The Jacobi polynomials,  denoted by $\{P_k^{\alpha,\beta}\}$ $(k=0,1,\ldots)$, are orthogonal with respect to the weight function
\begin{equation}
\label{Jacobi_weight_fn}
w_{\alpha,\beta}(x) := (1-x)^\alpha(1+x)^\beta, \quad x \in (-1,1)
\end{equation}
where $\alpha > -1$ and $\beta >-1$ are given parameters.
The normalized Jacobi polynomials are denoted by $\hat P^{\alpha,\beta}_k$, so that
$\int_{-1}^1 (\hat P^{\alpha,\beta}_k(x))^2 w_{\alpha,\beta}(x)dx =1$.

Thus the Chebyshev polynomials may be seen as the special case corresponding to $\alpha = \beta = -\frac{1}{2}$.

Likewise, the Legendre polynomials are the orthogonal polynomials w.r.t.\ the constant weight function ($w(x)=1$), so they correspond to the special case $\alpha = \beta = 0$.


There is no closed-form expression for the roots of Jacobi polynomials in general. But some bounds are known for the smallest root of $P_k^{\alpha,\beta}$, denoted by $\xi^{\alpha,\beta}_k$,  that we recall in the next theorem.

\begin{theorem}\label{thm_root}
The smallest root, denoted  $\xi^{\alpha,\beta}_k$, of the Jacobi polynomial $P^{\alpha,\beta}_k$ satisfies the following inequalities:
\begin{itemize}
\item[(i)] (\cite{Driver_Jordaan_2012})
$\xi^{\alpha,\beta}_k\le -1 + \frac{2(\beta+1)(\beta+3)}{2(k-1)(k+\alpha+\beta+2)+(\beta+3)(\alpha+\beta+2)}.$

\item[(ii)] (\cite{Dimitrov_Nikolov_2010})
$\xi^{\alpha,\beta}_k\ge \frac{F-4(k-1)\sqrt{\Delta}}{E},$ where
\begin{eqnarray*}
F &=& (\beta - \alpha)\left((\alpha + \beta + 6)k + 2(\alpha + \beta)\right), \\
E &=& (2k + \alpha + \beta)\left(k(2k + \alpha + \beta) + 2(\alpha + \beta + 2)\right) \\
\Delta &=&  k^2(k + \alpha + \beta + 1)^2 +(\alpha + 1)(\beta + 1)(k^2 + (\alpha + \beta + 4)k + 2(\alpha + \beta)).
\end{eqnarray*}
\end{itemize}
\end{theorem}
\ignore{
 may upper bounded by (see \cite{Driver_Jordaan_2012}):
\begin{equation}
\label{smallest_Jacobi_root_bound}
\end{equation}
In the case of Legendre polynomials ($\alpha = \beta = 0$) this bound becomes:
\[
-1 + \frac{3}{(k-1)(k+2)+3}.
\]
Moreover, the smallest root of $P_k^{\alpha,\beta}$ is lower bounded by (see \cite{Dimitrov_Nikolov_2010}):
\begin{equation}
\label{Jacobi_root_upper_bound}
\frac{F-4(k-1)\sqrt{\Delta}}{A},
\end{equation}
where
\begin{eqnarray*}
F &=& (\beta - \alpha)\left((\alpha + \beta + 6)k + 2(\alpha + \beta)\right), \\
A &=& (2k + \alpha + \beta)\left(k(2k + \alpha + \beta) + 2(\alpha + \beta + 2)\right) \\
\Delta &=&  k^2(k + \alpha + \beta + 1)^2 +(\alpha + 1)(\beta + 1)(k^2 + (\alpha + \beta + 4)k + 2(\alpha + \beta)).
\end{eqnarray*}
Note that $\frac{B-4(k-1)\sqrt{\Delta}}{A} = -1 +  \Omega\left(\frac{1}{k^2} \right)$.
For example, if $\alpha = \beta = 0$, one has
\begin{eqnarray*}
  \frac{B-4(k-1)\sqrt{\Delta}}{A} &=&   \frac{-4(k-1)\sqrt{k^2(k  + 1)^2 +k^2 + 4k  }}{2k\left(2k^2 + 4\right)}\\
   &\ge&  \frac{-4\sqrt{k^2(k^2  - 1)^2 +(k-1)^2(k^2 + 4k)  }}{4k^3 } \\
   &=& -\sqrt{1 - \frac{1}{k^2} + \frac{2}{k^3} - \frac{6}{k^4} + \frac{4}{k^5}} \\
   &=&  -1 + \Omega\left(\frac{1}{k^2}\right),
\end{eqnarray*}
where the last equality follows from the expansion
\[
\sqrt{1+x} = 1 +\frac{1}{2}x - \frac{1}{8}x^2 + \frac{1}{16}x^3 - \ldots
\]
}

The smallest roots $\xi^{\alpha,\beta}_k$ of the Jacobi polynomials $P^{\alpha,\beta}_k$ converge to $-1$ as $k\to \infty$. Using the above  bounds we see  that the rate of convergence  is
 $O(1/k^2)$.

\begin{corollary}
\label{cor_root}
The smallest roots of the Jacobi polynomials $P^{\alpha,\beta}_k$ satisfy
$$\xi^{\alpha,\beta}_k= -1 + \Theta\left({1\over k^2}\right) \ \  \text{ as } \ k\rightarrow \infty.$$
\end{corollary}

\begin{proof}
The upper bound in  Theorem \ref{thm_root}(i) gives directly $\xi^{\alpha,\beta}_k= -1 + O\left({1\over k^2}\right)$.
We now use the lower bound in Theorem \ref{thm_root}(ii) to show $\xi^{\alpha,\beta}_k= -1 + \Omega\left({1\over k^2}\right)$. For this we give asymptotic estimates for the quantities $E,F,\Delta$. First, using the expansion $\sqrt{1+x}= 1+{x\over 2}-{x^2\over 8}+o(x^2)$ as $x\rightarrow 0$ we obtain
$$\sqrt \Delta = k^2\left(1+{\alpha+\beta+1\over k} +{(\alpha+1)(\beta+1)\over 2k^2} +o\left({1\over k^2}\right)\right).$$
Second, using the expansion ${1\over 1+x}=1-x+x^2+o(x^2)$ as $x\rightarrow 0$ we obtain
$${1\over E}={1\over 4k^3}\left(1-{\alpha+\beta\over k} -{4(\alpha+\beta+2)\over k^2} +o\left({1\over k^2}\right)\right).$$
Combining these two relations gives
$$\begin{array}{lll}
{4(k-1)\sqrt\Delta\over E} &=
& \left(1-{1\over k}\right)
\left(1+{\alpha+\beta+1\over k} +{(\alpha+1)(\beta+1)\over 2k^2} +o\left({1\over k^2}\right)\right)
\left(1-{\alpha+\beta\over k} -{4(\alpha+\beta+2)\over k^2} +o\left({1\over k^2}\right)\right)\\
&=& 1 +{C\over 2k^2} +o\left({1\over k^2}\right),
\end{array}$$
where we set $C= (\alpha+1)(\beta+1) -8(\alpha+\beta+2) -2(\alpha+\beta)(\alpha+\beta+1)-2$.
Finally, using
$${F\over E}= {(\beta-\alpha)(\beta+\alpha+6)\over 4k^2} +o\left({1\over k^2}\right),$$
we obtain
$${F-4(k-1)\sqrt\Delta\over E}
= -1 + {1\over k^2}\left( {(\beta-\alpha)(\beta+\alpha+6)\over 4}  -{C\over 2} \right)+o\left({1\over k^2}\right),$$
where the coefficient of $1/k^2$ can be verified to be strictly positive, which thus implies the estimate $\xi^{\alpha,\beta}_k=-1 +\Omega(1/k^2)$.
\end{proof}

It is also known that $P^{\alpha,\beta}_k(x)=(-1)^k P^{\beta,\alpha}_k(-x)$. Therefore the largest root of $P^{\alpha,\beta}_k(x)$ is equal to $-\xi^{\beta,\alpha}_k= 1-\Theta(1/ k^2)$.

\section{{Tight lower bounds for a class of examples}}

In this section we consider the following simple examples
\begin{equation}
\label{example}
\min \left\{\sum_{i=1}^n c_ix_i: \;  \; x \in [-1,1]^n\right\},
\end{equation}
asking to minimize the  linear polynomial $f(x)=\sum_{i=1}^n c_ix_i$ over the box $\bK=[-1,1]^n$. Here $c_i\in \oR$ are given scalars for $i\in [n]$. Hence, $f_{\min,\bK}=-\sum_{i=1}^n |c_i|$.
For these examples we can obtain explicit closed-form expressions for the Lasserre bounds \smash{$\underline f^{(d)}_\bK$} when using product measures with  weight functions $w_{\alpha,\beta}$ on $[-1,1]$,   and also for the strengthened bounds \smash{$\uuf^{(d)}$} considered by De Klerk, Hess and Laurent, which use product measures with weight functions  $w_{-1/2,-1/2}$. These closed-form expressions are in terms of extremal roots of Jacobi polynomials.

\subsection{Tight lower bound for the Lasserre hierarchy}\label{secLas}

Here we consider the bounds \smash{$\underline{f}^{(d)}_{\mathbf{K}}$}  for the example (\ref{example}), when the measure $\mu$ on $\bK=[-1,1]^n$ is a product of univariate measures given by weight functions.

First we consider 
 the univariate case $n=1$.
When  the measure $\mu$  on $\mathbf{K} = [-1,1]$
is given by a continuous positive weight function $w$ on $(-1,1)$,
one can obtain a closed form expression for  $\underline{f}^{(d)}_{\mathbf{K}}$ in terms of the smallest root of the corresponding orthogonal polynomials.

\begin{theorem}\label{thm:main1}
 Consider the measure $d\mu(x) = w(x)dx$ on $\bK=[-1,1]$, where $w$ is a positive, continuous weight function on $(-1,1)$, and let $p_k$ be  univariate degree $k$ polynomials that are orthogonal
 with respect to this measure.
For the univariate polynomial $f(x)=x$ {(resp., $f(x)=-x$),}
the parameter  \smash{$\underline{f}^{(d)}_{\mathbf{K}}$}  is equal to the smallest root {(resp., the opposite of the largest root)} of the polynomial $p_{d+1}$.
\end{theorem}

\begin{proof}
Let $\hat p_0,\ldots,\hat p_{d+1}$  denote  the corresponding orthonormal polynomials, with $\hat p_i=p_i/\sqrt{\langle p_i,p_i\rangle}$.
Consider first $f(x)=x$. Using Lemma \ref{lemsummarize}, we see that \smash{$\underline{f}^{(d)}_{\mathbf{K}}$}  is equal to the smallest eigenvalue  of the  matrix $A$ in (\ref{eqhatA}) (for $k=d+1$),
which coincides with the matrix $A_{d+1}$ in (\ref{matrix_Ak}), so that its smallest eigenvalue is equal to the smallest root of $p_{d+1}$.\\
{Assume now $f(x)=-x$. Then \smash{$\underline{f}^{(d)}_{\mathbf{K}}$} is equal to $\lambda_{\min}(-A)=-\lambda_{\max}(A)$, which in turn is equal to the opposite of the largest root of $p_{d+1}$.}
\ignore{
As recalled earlier,  we may assume that $\hat p_0,\ldots,\hat p_{d+1}$ satisfy  three-terms  recurrence relations of the form
\eqref{eq:recursion}. This permits to compute the entries of the matrix $A$ as follows:
\begin{eqnarray*}
A_{ij} &=& \int_{-1}^1 x\hat p_i(x)\hat p_j(x) w(x)dx \\
       &=&  \int_{-1}^1\left(\alpha_i\hat p_{i+1}(x) + \beta_i\hat p_i(x) + \gamma_i\hat p_{i-1}(x)\right) \hat p_j(x)   w(x)dx \\
       &=& \left\{
       \begin{array}{rl}
       \alpha_i & \mbox{if $j=i+1$} \\
       \beta_i & \mbox{if $j=i$} \\
       \gamma_i & \mbox{if $j=i-1$} \\
       0 & \mbox{otherwise.}
\end{array}\right.
\end{eqnarray*}
Comparing with the matrix in \eqref{matrix_Ak}, one finds $A = A_{d+1}$. Hence,  the smallest eigenvalue of
$A$ is the smallest root of $p_{d+1}$, by \eqref{eq:det_roots}.}
\end{proof}

Recall that $\xi^{\alpha,\beta}_{d+1}$ denotes the smallest root of the Jacobi polynomial $P^{\alpha,\beta}_{d+1}$ and that the largest root of $P^{\alpha,\beta}_{d+1}$ is equal to $-\xi^{\beta,\alpha}_{d+1}$.

\begin{corollary}\label{corboundn1}
Consider  the measure $d\mu(x)=w_{\alpha,\beta}(x)dx$ on $\bK=[-1,1]$ with the weight function $w_{\alpha,\beta}(x) = (1-x)^\alpha(1+x)^\beta$ and $\alpha,\beta>-1$.
For the univariate polynomial $f(x)=x$ {(resp., $f(x)=-x$)},  the parameter
$\smash{\underline{f}^{(d)}_{\mathbf{K}}}$ is equal to \smash{$ \xi^{\alpha,\beta}_{d+1}$}
  {(resp., to \smash{$\xi^{\beta,\alpha}_{d+1}$)}} and thus we have
  $$\underline{f}^{(d)}_{\mathbf{K}}- f_{\min,\bK}=\Theta(1/d^2).$$
  In particular,  \smash{$\underline{f}^{(d)}_{\mathbf{K}}=-\cos\left( \frac{\pi}{2d+2} \right)$}
when $\alpha=\beta=-1/2$.
 \end{corollary}

\begin{proof}
This follows directly using Theorem \ref{thm:main1}, Corollary \ref{cor_root}, the fact that the largest root of $P^{\alpha,\beta}_{d+1}$ is equal to \smash{$-\xi^{\beta,\alpha}_{d+1}$,}  and  the closed form expression (\ref{rootTk}) for the roots of  the  Chebyshev polynomials of the first kind.
\end{proof}

We now use the  above result to show  $\underline{f}^{(d)}_{\mathbf{K}}- f_{\min,\bK} = \Omega(1/d^2)$ for  the example (\ref{example}) in the multivariate case $n\ge 2$.

\begin{corollary}\label{corboundn}
Consider the measure $d\mu(x) =\prod_{i=1}^n w_{\alpha_i,\beta_i}(x_i) dx_i$ on the hypercube  $\bK=[-1,1]^n$, with the weight functions $w_{\alpha_i,\beta_i}(x_i)=(1-x_i)^{\alpha_i}(1+x_i)^{\beta_i}$ and $\alpha_i,\beta_i>-1$ for $i\in [n]$.
For the polynomial $f(x)=\sum_{l=1}^n c_l x_l$, we have
$$\underline{f}^{(d)}_{\bK}\ge  \sum_{l: c_l>0} c_l\xi^{\alpha_l,\beta_l}_{d+1} +\sum_{l:c_l<0}|c_l| \xi^{\beta_l,\alpha_l}_{d+1},$$
and thus    $\underline{f}^{(d)}_{\mathbf{K}} -f_{\min,\bK}= \Omega(1/d^2)$.
\end{corollary}

\begin{proof}
Assume $\underline{f}^{(d)}_{\mathbf{K}}=\int_\bK (\sum_{l=1}^n {c_l} x_l)\sigma(x)d\mu(x)$, where $\sigma\in \oR[x_1,\ldots,x_n]_{2d}$ is a sum of squares of polynomials and $\int_\bK \sigma(x)d\mu(x)=1$.
For each $l\in [n]$ consider the univariate polynomial
$$\sigma_l(x_l):= \int_{[-1,1]^{n-1}} \sigma(x_1,\ldots,x_n)\prod_{i\in [n]\setminus \{l\}} w_{\alpha_i,\beta_i}(x_i) dx_i,$$
where we integrate over all variables $x_i$ with $i\in[n]\setminus \{l\}$. Then we have  $\int_{-1}^1\sigma_l(x_l)w_{\alpha_l,\beta_l}(x_l)dx_l=1$. Moreover, $\sigma_l$ has degree at most $2d$ and, as it is a univariate polynomial  which is nonnegative on $\oR$, it  is a sum of squares of polynomials.
Hence, using Corollary \ref{corboundn1}, we can conclude that
$$\int_{-1}^1 x_l\sigma_l(x_l)w_{\alpha_l,\beta_l}(x_l)dx_l \ge \xi^{\alpha_l,\beta_l}_{d+1},\quad
{\int_{-1}^1 (-x_l)\sigma_l(x_l)w_{\alpha_l,\beta_l}(x_l)dx_l \ge \xi^{\beta_l,\alpha_l}_{d+1}.}
$$
Combining with the definition of $\underline{f}^{(d)}_{\mathbf{K}}$ we obtain
$$\underline{f}^{(d)}_{\mathbf{K}} =\sum_{l=1}^n {c_l}\int_{-1}^1 {x_l}\sigma_l(x_l) w_{\alpha_l,\beta_l}(x_l)dx_l \ge
{\sum_{l:c_l>0} c_l  \xi^{\alpha_l,\beta_l}_{d+1}+\sum_{l: c_l<0} |c_l| \xi^{\beta_l,\alpha_l}_{d+1}}$$
and thus $\underline{f}^{(d)}_{\mathbf{K}} -f_{\min,\bK}
 \ge { \sum_{l: c_l>0} c_l(\xi^{\alpha_l,\beta_l}_{d+1}+1) +\sum_{l: c_l<0} |c_l|  (\xi^{\beta_l,\alpha_l}_{d+1}+1)}=\Omega(1/d^2).$
\end{proof}

\subsection{Tight lower bound  for the De Klerk, Hess and Laurent hierarchy}\label{secDKHL}

In this section  we consider the  hierarchy of bounds \smash{$\uuf^{(d)}$} studied by  De Klerk, Hess and Laurent \cite{DHL SIOPT}, 
which are potentially stronger than the bounds \smash{$\underline f^{(d)}_\bK$} since they involve the wider class of density functions   in (\ref{eqSch}). Their convergence rate is known to be $O(1/d^2)$ (\cite{DHL SIOPT}, recall Theorem \ref{theoDKHL}).

For the  example \eqref{example}  we can also give an explicit expression for the bounds \smash{$\uuf^{(d)}$} and we will show that their convergence rate to $f_{\min,\bK}$ is also in the order $\Omega(1/d^2)$, which shows that the analysis in \cite{DHL SIOPT} is tight.

We first treat the univariate case, in order to introduce the main ideas, and then we extend to the multivariate case.

\begin{theorem}\label{thm:rateDKHL}
For the univariate polynomial $f(x)={\pm} x$,  we have
$$\uuf^{(d)} =\min\{\xi^{-1/2,-1/2}_{d+1},\xi^{1/2,1/2}_d\},$$ the smallest value among the smallest roots of the Jacobi polynomials $P^{-1/2,-1/2}_{d+1}$ and  $P^{1/2,1/2}_d$. In particular, we have
 $\uuf^{(d)}-f_{\min,\bK}= \Theta(1/d^2)$.
\end{theorem}

\begin{proof}
Consider first $f(x)=x.$ We first recall how to compute  \smash{$\uuf^{(d)}$} as an eigenvalue problem.
 By definition,
it is the minimum value of $\int_{-1}^1 x (\sigma_0(x)+\sigma_1(x)(1-x^2))w_{-1/2,-1/2}(x)dx$,
where $\sigma_0\in \Sigma[x]_{2d}$, $\sigma_1\in \Sigma[x]_{2d-2}$ and
$\int_{-1}^1 (\sigma_0(x)+\sigma_1(x)(1-x^2))w_{-1/2,-1/2}(x)dx=1$.
We express the polynomial $\sigma_0$ in the normalized Jacobi (Chebychev) basis {$\{\hat P^{-1/2,-1/2}_k\}$} as
$$\sigma_0=\sum_{i,j=0}^d M^{(0)}_{ij} \hat P^{-1/2,-1/2}_i \hat P^{-1/2,-1/2}_j$$ for some matrix $M^{(0)}$ of order $d+1$, constrained to be positive semidefinite.
Based on the  observation  that
$(1-x^2) w_{-1/2,-1/2}(x)=w_{1/2,1/2}(x)$, we express the polynomial $\sigma_1$ in the normalized Jacobi basis {$\{\hat P^{1/2,1/2}_k\}$} as
$$\sigma_1=\sum_{i,j=0}^{d-1} M^{(1)}_{ij} \hat P^{1/2,1/2}_i \hat P^{1/2,1/2}_j$$
for some matrix $M^{(1)}$ of order $d$, also constrained to be positive semidefinite.
Then, we obtain
$$\uuf^{(d)}= \min \{\langle A^{-1/2,-1/2}_d,M^{(0)}\rangle +\langle A^{1/2,1/2}_{d-1},M^{(1)}\rangle: \text{\rm Tr}(M^{(0)})+\text{\rm Tr}(M^{(1)})=1,\ M^{(0)}\succeq 0, M^{(1)}\succeq 0\},$$
{where $A^{1/2,1/2}_d$ and $A^{-1/2,-1/2}_{d-1}$ are instances of (\ref{eqhatA}) defined as follows: }
\begin{equation*}\label{eqAab}
A^{\alpha,\beta}_d:=\left( \int_{-1}^1x \hat P^{\alpha,\beta}_h(x)  \hat P^{\alpha,\beta}_k(x) w_{\alpha,\beta}(x) dx\right)_{h,k=0}^d
\end{equation*}
{for any $\alpha,\beta>-1$ and $d\in\oN$.}
Since strong duality holds  we obtain
$$\uuf^{(d)} =\max\{t: A^{-1/2,-1/2}_d-tI\succeq 0,\ A^{1/2,1/2}_{d-1} -tI\succeq 0\}= \min \{\lambda_{\min}(A^{-1/2,-1/2}_d),\lambda_{\min}(A^{1/2,1/2}_{d-1})\}.$$
By Lemma \ref{lem_Ak}, we have $\lambda_{\min}(A^{-1/2,-1/2}_d)=\xi^{-1/2,-1/2}_{d+1}$
and $\lambda_{\min}(A^{1/2,1/2}_{d-1})=\xi^{1/2,1/2}_d$ and thus
$\uuf^{(d)} = \min\{\xi^{-1/2,-1/2}_{d+1},\xi^{1/2,1/2}_d\}$. {The same result holds when $f(x)=-x$.}
Finally, by Corollary \ref{cor_root}, these two smallest roots are both equal to  $-1+\Theta(1/d^2)$, which concludes the proof.
\end{proof}

We now extend this result to   the multivariate case of example \eqref{example}:

\begin{corollary}\label{corbounddKHLn}
For the linear polynomial $f(x)=\sum_{l=1}^n c_lx_l$, we have $$\uuf^{(d)} \ge {\left(\sum_{l=1}^n|c_l|\right)} \min\{\xi^{-1/2,-1/2}_{d+1},\xi^{1/2,1/2}_d\}$$ and thus
$\uuf^{(d)} -f_{\min,\bK}= \Omega(1/d^2).$

\end{corollary}

\begin{proof}
The proof is analogous to that of Corollary \ref{corboundn},  with some more technical details.
Assume $\uuf^{(d)}=\int_\bK (\sum_{l=1}^n x_l)\sigma(x)d\mu(x)$, where $\sigma(x)=\sum_{I\subseteq [n]}\sigma_I(x)\prod_{i\in I}(1-x_i^2)$, $\sigma_I(x)$ is a sum of squares of degree at most $2d-2|I|$ and $\int_\bK \sigma(x)d\mu(x)=1$.

Fix $l\in [n]$. Then we can write
$$\sigma(x)=  \sum_{I\subseteq [n]\setminus\{l\}} \sigma_I(x)\prod_{i\in I}(1-x_i^2)
\ + \ (1-x_l^2)
  \sum_{I\subseteq [n]: l\in I} \sigma_I(x)\prod_{i\in I\setminus\{l\}}(1-x_i^2).
$$
Next,  define the univariate polynomials  in the variable $x_l$:
$$\sigma_{l,0}(x_l):=\sum_{I\subseteq [n]\setminus\{l\}} \int_{[-1,1]^{n-1}} \sigma_I(x)\prod_{i\in I}(1-x_i^2) \prod_{i\in [n]\setminus \{l\}} w_{-1/2,-1/2}(x_i)dx_i,$$
$$\sigma_{l,1}(x_l):=\sum_{I\subseteq [n]: l\in I} \int_{[-1,1]^{n-1}} \sigma_I(x)\prod_{i\in I{\setminus\{l\}}}(1-x_i^2) \prod_{i\in [n]\setminus \{l\}} w_{-1/2,-1/2}(x_i)dx_i,$$
$$\sigma_l(x_l):= \int_{[-1,1]^{n-1}} \sigma(x) \prod_{i\in [n]\setminus\{l\}} w_{-1/2,-1/2}(x_i)dx_i
= \sigma_{l,0}(x_l)+ (1-x_l^2)\sigma_{l,1}(x_l).$$
By construction, we have
$$\int_\bK x_l\sigma(x)d\mu(x)= \int_{-1}^1 x_l\sigma_l(x_l)w_{-1/2,-1/2}(x_l)dx_l,\ \
\int_{-1}^1 \sigma_l(x_l) w_{-1/2,-1/2}(x_l)dx_l= \int_\bK \sigma(x)d\mu(x)=1.$$
Moreover, the polynomial $\sigma_{l,0}$ is a sum of squares (since it is  univariate and nonnegative on $\oR$) and its degree is at most $2d$,
and the polynomial $\sigma_{l,1}$ is a sum of squares of degree at most $2d-2$.
Hence, using Theorem \ref{thm:rateDKHL}, we can conclude that
$$\int_{-1}^1 ({\pm} x_l)\sigma_l(x_l)w_{-1/2,-1/2}(x_l)dx_l \ge \min\{\xi^{-1/2,-1/2}_{d+1},\xi^{1/2,1/2}_d\}.$$
This implies that
$$
\uuf^{(d)} =\int_\bK (\sum_{l=1}^nc_l x_l)\sigma(x)d\mu(x)
= \sum_{l=1}^n c_l \int_{-1}^1 x_l\sigma_l(x_l)w_{-1/2,-1/2}(x_l)dx_l $$ is at least ${ (\sum_l|c_l|)} \min\{\xi^{-1/2,-1/2}_{d+1},\xi^{1/2,1/2}_d\}$
and the proof is complete.
\end{proof}

\section{Tight upper bounds  for the Lasserre hierarchy} \label{secupper}

In this section we analyze the rate of convergence of the Lasserre bounds $\underline{f}^{(d)}_{\mathbf{K}}$
when using the measure $d\mu(x)=\prod_{i=1}^n w_{-1/2,-1/2}(x_i)dx_i$ on the box $\bK=[-1,1]^n$ (corresponding to the Chebyshev orthogonal polynomials). For this measure, it is known that the stronger bounds $f^{(d)}$ - that use a much richer class of density functions -
enjoy a $O(1/d^2)$ rate of convergence (\cite{DHL SIOPT}, see Theorem \ref{theoDKHL}). We show that the convergence rate remains $O(1/d^2)$ for the weaker bounds $\underline{f}^{(d)}_{\mathbf{K}}$, which thus also implies Thoerem \ref{theoDKHL}.

\begin{theorem}\label{theoLasrate}
Consider the measure $d\mu(x)=\prod_{i=1}^n w_{-1/2,-1/2}(x_i)dx_i$ on the hypercube $\bK=[-1,1]^n$, with the weight function $w_{-1/2,-1/2}(x_i)=(1-x_i^2)^{-1/2}$ for $i\in [n]$. For any polynomial $f$ we have
$$\underline{f}^{(d)}_{\mathbf{K}} -f_{\min,\bK} =O(1/d^2).$$
\end{theorem}

\noindent
It  turns out that we can reduce the general result to the univariate quadratic case.
In what follows we consider first  the special case when $f$ is univariate  and quadratic (see Lemma \ref{lemquaduni}) and then we indicate how to derive the result for an arbitrary multivariate polynomial $f$. A key tool we use for this reduction is the existence of a quadratic upper estimator for $f$ having the same minimum as $f$ over $\bK$.
In the quadratic univariate case we exploit  again the formulation of $\underline{f}^{(d)}_{\mathbf{K}}$ in terms of the smallest eigenvalue of the associated matrix $A_d$ in  (\ref{eqAdLas}) (recall Lemma \ref{lemsummarize}). This matrix $A_d$ is now 5-diagonal, but a key feature is that it contains a large Toeplitz submatrix, whose eigenvalues can be estimated by embedding it into a circulant matrix for which closed form expressions exist for the eigenvalues. This nice structure, which allows a simple analysis,  follows from the choice of the Chebyshev type measure. We expect that a similar
convergence rate should hold when selecting any measure of Jacobi type, but the analysis seems more complicated.

\subsection{The quadratic univariate case}

Here we consider  the case when $\bK=[-1,1]$ and $f$ is a univariate quadratic polynomial of the form
$f(x)=x^2+\alpha x$, for some scalar $\alpha\in\oR$.

We can first easily deal with the case when $\alpha\not\in (-2,2)$.  Indeed then we have
$$f(x)\le g(x):= \alpha x+1 \quad \text{ for all } x\in [-1,1],$$
and both $f$ and $g$ have the same minimum value on $[-1,1]$. Namely, $f_{\min,\bK}=g_{\min,\bK}$ is equal to $1-\alpha$ if $\alpha \ge 2$, and to $1+\alpha$ if $\alpha \le -2$. Therefore we have
$$\underline{f}^{(d)}_{\mathbf{K}} -f_{\min,\bK}\le \underline{g}^{(d)}_{\mathbf{K}} - g_{\min,\bK} =O(1/d^2),$$
where we use  Corollary \ref{corboundn} for the last estimate.

\medskip We may now assume that $f(x)=x^2+\alpha x$, where $\alpha\in [-2,2]$. Then, $f_{\min,\bK}= -\alpha^2/4$, which is  attained at $x=-\alpha /2$.
After scaling  the measure $\mu$ by $2/\pi$,
the Chebyshev polynomials $T_i$  satisfy
$$\int_{-1}^1 T_i(x)T_j(x) {2\over \pi \sqrt {1-x^2}} dx= 0 \text{ if } i\ne j, \ 2 \text{ if } i=j=0,\ 1 \text{ if } i=j\ge 1.$$
So with respect to this scaled measure the normalized Chebyshev polynomials are $\hat T_0=1/\sqrt 2$ and $\hat T_i=T_i$ for $i\ge 1$, and they satisfy the 3-terms relation:
$$x\hat T_1= {1\over 2}\hat T_2+{1\over \sqrt 2}\hat T_0 \quad \text{ and } \quad x \hat T_{k} ={1\over 2}\hat T_{k+1} +{1\over 2}\hat T_{k-1} \ \text{ for } k\ge 2.$$
In view of Lemma \ref{lemsummarize} we know that the parameter $\underline{f}^{(d)}_{\mathbf{K}}$  is equal to the smallest eigenvalue of the following matrix
$$A_d =\left(\int_{-1}^1 (x^2+\alpha x) \hat T_i(x)\hat T_j(x) {2\over \pi \sqrt {1-x^2}} dx\right)_{i,j=0}^d.$$
Using the above 3-terms relations one can verify that the matrix $A_d$ has the following form:

  \begin{equation}\label{eqAdLas}
A_d=\left(\begin{matrix}
{1\over 2} & {\alpha\over \sqrt 2} & {1\over 2\sqrt 2} & &&&&& \cr
{\alpha\over \sqrt 2} & {3\over 4} & {\alpha\over 2} & {1\over 4} & &&&& \cr
{1\over 2\sqrt 2} & {\alpha\over 2} & a& b & c & &&&\cr
                & {1\over 4} & b & a& b& c & && \cr
                & & c &b & \ddots & \ddots & \ddots & &\cr
                &&& c & \ddots & \ddots & \ddots & \ddots & \cr
               & &&&\ddots & \ddots & \ddots & \ddots  & c\cr
               & &&&&\ddots& \ddots & \ddots  & b \cr
                &&&&&& c& b & a
                \end{matrix}\right),
                 \end{equation}
where we set $a=1/2$, $b=\alpha/2$ and $c=1/4$.

Observe that if we remove the first two rows and columns of $A$ then we obtain a principal submatrix, denoted  $B$, which
  is a symmetric 5-diagonal Toeplitz matrix.
 Now we may embed $B$ into a symmetric circulant matrix of size $d+1$,
 denoted $C_d$, by suitably defining the first two rows and columns. Namely,
     $$
C_d=\left(\begin{matrix}
a & b & c & &&&&c& b\cr
b & a & b & c & &&&&c \cr
c & b & a & b & c & &&&\cr
                & c & b & a & b &c & && \cr
                & & c & b & \ddots & \ddots & \ddots & &\cr
                &&& c & \ddots & \ddots & \ddots & \ddots & \cr
               & &&&\ddots & \ddots & \ddots & \ddots  & c\cr
              c& &&&&\ddots& \ddots & \ddots  & b \cr
               b& c &&&&& c & b & a
                \end{matrix}\right).
                 $$
Recall that the eigenvalues of a circulant matrix are known in closed form, see, e.g.,\ \cite{circulant matrices}.
In particular, the eigenvalues of $C_d$ are given by
\begin{equation}
\label{eigs symmetric circulant}
 a + 2b\cos(2\pi j/(d+1)) + 2c\cos(2\pi 2j/(d+1), \quad j = 0,\ldots,d, \quad (d \ge 5).
\end{equation}

                 By the Cauchy interlacing theorem for eigenvalues (see, e.g.,\ Corollary 2.2 in \cite{Haemers95}),   we have
$$\underline{f}^{(d)}_{\mathbf{K}}=\lambda_{\min}(A_d)\le \lambda_{\min}(B) \le \lambda_3(C_d),$$ where $\lambda_3(C_d)$ is the third smallest eigenvalue of $C_d$.
{As noted above the eigenvalues of $C_d$ are known in closed form as in (\ref{eigs symmetric circulant}) and this is the key fact which enables us to conclude the analysis.}

\begin{lemma}\label{lemquaduni}
For any  $\alpha\in [-2,2]$, the third smallest eigenvalue of the matrix $C_d$ satisfies
$$\lambda_3(C_d)= -{\alpha^2\over 4} +O\left({1\over d^2}\right).$$
Therefore, if $f(x)=x^2+\alpha x$ with $\alpha \in [-2,2]$ then $\underline{f}^{(d)}_{\mathbf{K}}-f_{\min,\bK}=O(1/d^2)$.
\end{lemma}

\begin{proof}
Setting $\vartheta_j= {2\pi j \over d+1}$ for $j\in\oN$, then by \eqref{eigs symmetric circulant} the  eigenvalues of the matrix $C_d$ are
the scalars
$${1\over 2} +\alpha \cos (\vartheta_j) +{1\over 2} \cos(2\vartheta_j) = \cos^2(\vartheta_j)+\alpha \cos( \vartheta_j )\quad \text{ for } 0\le j\le d.$$
 Consider the function $f(\vartheta)= \cos^2(\vartheta)+\alpha \cos (\vartheta)$ for $\vartheta\in [0,2\pi]$. Then $f$ satisfies: $f(\vartheta)=f(2\pi-\vartheta)$, and its minimum value is equal to $-\alpha^2/4$, which is attained at $\vartheta= \arccos(-\alpha/2)\in [0,\pi]$ and $2\pi-\vartheta$.
  Let $j$ be the integer  such that $\vartheta_j \le \vartheta < \vartheta_{j+1}.$
 Then the smallest eigenvalue of $C_d$ is $\lambda_{\min}(C_d)=\min\{f(\vartheta_j), f(\vartheta_{j+1})\}$ and its  third smallest eigenvalue  is  given by $\lambda_3(C_d)= \min\{f(\vartheta_{j-1}), f(\vartheta_{j+1})\}$ if $\lambda_{\min}(C_d)= f(\vartheta_j)$,
 and
  $\lambda_3(C_d)= \min\{f(\vartheta_{j}), f(\vartheta_{j+2})\}$ if $\lambda_{\min}(C_d)= f(\vartheta_{j+1})$.
 Therefore, $\lambda_3(C_d)=f(\vartheta_k)$ for some $k\in \{ j-1,j,j+1,j+2\}$.

 Using Taylor theorem (and the fact that $f'(\vartheta)=0$) we can conclude that
 $$\lambda_3(C_d)+{\alpha^2\over 4}= f(\vartheta_k)-f(\vartheta)= {1\over 2}f''(\xi)(\vartheta-\vartheta_k)^2,$$
 for some scalar $\xi \in (\vartheta,\vartheta_k)$ (or $(\vartheta_k,\vartheta)$).
 Finally, $f''(\xi)= -2\cos(\xi) -\alpha \cos (\xi)$ and thus we have $|f''(\xi)|\le 2+|\alpha|$.
 Also $|\vartheta-\vartheta_k|\le |\vartheta_{j+2}-\vartheta_{j-1}| = {6\pi\over d+1}.$ The claimed result now follows directly. \end{proof}


\subsection{The general case}

As a direct application we can also deal with the case when $f$ is multivariate quadratic and separable.

 \begin{corollary} \label{corquadsep}
 Consider the box $\bK=[-1,1]^n$ and a multivariate polynomial of the form
 $f(x)=\sum_{i=1}^n x_i^2 +\alpha_i x_i$ for some scalars $\alpha_i\in \oR$.
 Then we have
 $\underline{f}^{(d)}_{\mathbf{K}}-f_{\min,\bK}=O(1/d^2)$.
 \end{corollary}

 \begin{proof}
 The polynomial $f$ is separable: $f(x)=\sum_{i=1}^n f_i(x_i)$, after setting $f_i(x_i)= x_i^2+\alpha_i x_i.$
Hence its minimum over the box $\bK$ is $f_{\min,\bK}=\sum_{i=1}^n (f_i)_{\min,[-1,1]}$.
 Suppose $\sigma_i\in \Sigma[x_i]_d$ is an optimal density function for the bound $\underline{f_i}^{(d)}_{[-1,1]}$ and consider the polynomial
 $\sigma(x)=\prod_{i=1}^n \sigma_i(x_i) \in \Sigma[x]_{nd},$ which is a density function over $\bK$. Then we have
 $$\underline{f}^{(nd)}_{\mathbf{K}} -f_{\min,\bK} \le \int_{\bK} f(x)\sigma(x)d\mu(x) =\sum_{i=1}^n \left(\int_{-1}^1 f_i(x_i) d\mu(x_i) -(f_i)_{\min,[-1,1]}\right) =O(1/d^2),$$
 where we use Lemma \ref{lemquaduni} for the last estimate.
  This implies the claimed convergence rate for the bounds $\underline{f}^{(d)}_{\mathbf{K}}$.
 \end{proof}

Assume now $f$ is  an arbitrary polynomial and let $a\in \bK=[-1,1]^n$ be a minimizer of $f$ over $\bK$.
 Consider the following quadratic polynomial
 $$g(x) = f(a)+ \nabla f(a)^T (x-a) + C_f \|x-a\|_2^2,$$
 where we set $C_f= \max_{x\in \bK} \|\nabla^2 f(x)\|_2$. By Taylor's theorem we know that $f(x)\le g(x)$ for all $x\in \bK$ and that the minimum value of $g(x)$ over $\bK$ is  $g_{\min,\bK}=f(a)=f_{\min,\bK}$.
 This implies
 $$\underline{f}^{(d)}_{\mathbf{K}}-f_{\min,\bK}\le  \underline{g}^{(d)}_{\mathbf{K}}-g_{\min,\bK}  =O(1/d^2),$$
 where we use Corollary \ref{corquadsep} for the last estimate.
 This concludes the proof of Theorem \ref{theoLasrate}.

\section{Concluding remarks}
Some other hierarchical upper bounds for polynomial optimization over the hypercube have been
investigated in the literature. In particular, bounds are proposed in \cite{KLLS MOR}, that rely on selecting
 density functions arising from beta distributions:
 \[
f^H_d:=\,\displaystyle\min_{(\alpha,\beta)\in {N}(2n,d)}\:\frac{\displaystyle\int_\bK f(x)\,x^\alpha(1-x)^\beta\,dx}
{\displaystyle\int_\bK x^\alpha(1-x)^\beta\,dx},
\]
where, $\bK = [-1,1]^n$, and $(1-x)^\beta = \prod_{i=1}^n (1-x_i)^{\beta_i}$ for $\beta \in \mathbb{N}^n$.
 These bounds can be computed via elementary operations only and
 their rate of convergence is  $f^H_d-f_{\min,\bK}= O(1/\sqrt d)$ (or $O(1/d)$ for quadratic polynomials with rational data).

 Other hierarchies involve
  selecting discrete measures. They rely on polynomial evaluations at rational grid points \cite{KL SIOPT} or at polynomial
   meshes like Chebyshev grids \cite{PV OL}.
   The grids in \cite{PV OL} are given by the Chebyshev-Lobatto points:
   \[
   C_d := \left\{\cos\left(\frac{j\pi}{d}\right)  \right\} \quad j = 0,\ldots,d.
   \]
   In particular the authors of \cite{PV OL} show that $\min_{x \in C_d^n} f(x)  -f_{\min,\bK} = O\left(\frac{1}{d^2}\right)$, where
   $$ C_d^n= C_d \times \cdots \times C_d \subset [-1,1]^n.$$
   Note that $|C_d^n| = (d+1)^n$, which is of course exponential in $n$ even for fixed $d$.

The same {$O\left(\frac{1}{d^2}\right)$}  rate of convergence was shown in  \cite{KL SIOPT} for the regular grid ({using} $d+1$ evenly spaced points).
{We also refer to the recent work \cite{PV18} where polynomial meshes are investigated for polynomial optimization over general convex bodies.}

Thus the Lasserre bound $\underline{f}^{(d)}_{\mathbf{K}}$ has the same $O\left(\frac{1}{d^2}\right)$ asymptotic rate of convergence as the grid searches, but with the advantage
that the computation may be done in polynomial time for fixed $d$.

Of course, the problem studied in this paper falls in the general framework of bound-constrained global optimization problems, and many other
algorithms are available for such problems; a recent survey is given in the thesis \cite{Pal thesis}. The point is that the

 methods we studied in this paper allow analysis of the convergence rate to the global minimum.

We conclude with some unresolved questions:
\begin{itemize}
\item
Does the $O\left(\frac{1}{d^2}\right)$ rate of convergence still hold for the Lasserre bounds if $\bK$ is a general convex body?
(The best known result is the $O(1/d)$ rate from \cite{DKL MOR}.)
\item
What is the precise influence of the choice of reference measure $\mu$ in \eqref{fminkreform2} on the convergence rate?
\item
Is is possible to show a `saturation' result for the Lasserre bounds of the type:
\[
\underline{f}^{(d)}_{\mathbf{K}} - f_{\min,\bK} = o\left( \frac{1}{d^2}\right) \Longleftrightarrow  \mbox{ $f$ is a constant polynomial?}
\]
In other words, is $O(1/d^2)$ the fastest possible convergence rate for nonconstant polynomials?
\end{itemize}

\vspace*{0.5cm}\noindent
{\bf Acknowledgements.}
The authors would like to thank Jean-Bernard Lasserre for useful discussions.


\begin{thebibliography}{10}


%
%

\bibitem{KL SIOPT}	
E. de Klerk, M. Laurent.
Error bounds for some semidefinite programming approaches to polynomial optimization on the hypercube.
{\em SIAM Journal on  Optimization}
20(6), (2010) 3104--3120.



\bibitem{DKL MOR}
E. de Klerk and M. Laurent.
Comparison of Lasserre's measure-based bounds for polynomial optimization to bounds obtained by simulated annealing. \emph{Mathematics of Operations Research}, to appear.  	arXiv:1703.00744

\bibitem{DHL SIOPT}
E. de Klerk, R. Hess and M. Laurent.
Improved convergence rates for Lasserre-type hierarchies of upper bounds for box-constrained polynomial optimization.
{\em SIAM Journal on Optimization}
 27(1), (2017) 347-367.

\bibitem{KLLS MOR}
E. de Klerk, J.-B. Lasserre, M. Laurent, and Z. Sun.
Bound-constrained polynomial optimization using only elementary calculations.
{\em Mathematics of Operations Research} 42(3), (2017) 834--853.  	
	
	


\bibitem{KLS MPA}
E. de Klerk, M. Laurent, Z. Sun.
Convergence analysis for Lasserre's measure-based hierarchy of upper bounds for polynomial optimization,
\textit{Mathematical Programming Ser. A} 162(1), (2017)  363-392.



\bibitem{Dimitrov_Nikolov_2010}
D.K. Dimitrov, G.P. Nikolov.
 Sharp bounds for the extreme zeros of classical orthogonal polynomials, \emph{Journal of Approximation Theory} 162 (2010), 1793--1804.

\bibitem{Driver_Jordaan_2012}
K. Driver, K. Jordaan.
Bounds for extreme zeros of some classical
orthogonal polynomials.
\emph{Journal of Approximation Theory} 164 (2012), 1200--1204.

\bibitem{Gautsch}
W. Gautsch.
{\em Orthogonal Polynomials - Computation and Approximation}.
Oxford University Press, 2004.


\bibitem{circulant matrices}
R.M.\ Gray. Toeplitz and circulant matrices: A review.
\emph{Foundations and Trends in Communications and Information Theory} {2}(3) (2006), 155--239.

\bibitem{Haemers95}
W.H.\ Haemers.
Interlacing eigenvalues and graphs.
\emph{Linear Algebra and its Applications} 227-228 (1995), 593--616.

\bibitem{Ismail_Li_1992}
M.E.H. Ismail, X. Li.
Bounds on the extreme zeros of orthogonal polynomials.
 \emph{Proc. Amer. Math. Soc.} 115 (1992),
131--140.

\bibitem{Las11}
J.B. Lasserre.
A new look at nonnegativity on closed sets and polynomial optimization.  {\em SIAM
 Journal on Optimization} 21(3) (2011), 864--885.

\bibitem{Pal thesis}
P\'al, L.:
{\em Global Optimization Algorithms for Bound Constrained Problems.} PhD thesis, University of Szeged (2010). Available at \url{http://www2.sci.u-szeged.hu/fokozatok/PDF/Pal_Laszlo/Diszertacio_PalLaszlo.pdf}


\bibitem{PV OL}
F. Piazzon, M. Vianello.
A note on total degree polynomial optimization by Chebyshev grids.
{\em Optimization Letters} 12 (2018), 63--71.

\bibitem{PV18}
F. Piazzon and M. Vianello.
Markov inequalities, Dubiner distance, norming meshes and polynomial optimization on convex bodies.
Preprint, 2018. Available at \url{http://www.math.unipd.it/~marcov/pdf/convbodies.pdf}

\bibitem{Szego_1975}
G. Szeg\"o.
\emph{Orthogonal Polynomials, fourth ed.}, vol. XXIII, American Mathematical Society Colloquium
Publications, Providence, RI, 1975.


\end{thebibliography}
 \end{document}